\documentclass{amsart}
\usepackage{amssymb,color,xcolor}
\usepackage{eucal}
\usepackage{esint}
\usepackage{hyperref}
\usepackage{ifpdf}   
\theoremstyle{plain}
\newtheorem{theorem}{Theorem}

\newtheorem{lemma}{Lemma}

\theoremstyle{definition}
\newtheorem{definition}{Definition}
\theoremstyle{remark}

\numberwithin{equation}{section}
\newcommand{\e}{\epsilon}

\newcommand{\R}{\mathbb R}

\newcommand{\Lc}{\mathcal{L}}

\newcommand{\xb}{\overline{x}}
\newcommand{\yb}{\overline{y}}

\newcommand{\rst}[1]{\ensuremath{{\mathbin\upharpoonright}%
\raise-.5ex\hbox{$#1$}}}

\def\p{\partial}

\def \e {\varepsilon}

\def \Om {\Omega}

\begin{document}

\title[$L^p$ regularity and the Dini condition]{Some remarks on the $L^p$ regularity of second derivatives of solutions to non-divergence elliptic equations and the Dini condition.}
%
\author{Luis Escauriaza}
\address[Luis Escauriaza]{Universidad del Pa{\'\i}s Vasco/Euskal Herriko
Unibertsitatea\\Dpto. de Matem{\'a}ticas\\Apto. 644, 48080 Bilbao, Spain.}
\email{luis.escauriaza@ehu.eus}

\author{Santiago Montaner}
\address[Santiago Montaner]{Universidad del Pa{\'\i}s Vasco/Euskal Herriko
Unibertsitatea\\Dpto. de Matem{\'a}ticas\\Apto. 644, 48080 Bilbao, Spain.}
\email{santiago.montaner@ehu.eus}
\thanks{The authors are supported  by the grants MTM2014-53145-P and IT641-13 (GIC12/96).}

\keywords{regularity, pathological solutions, non divergence form elliptic operator}
\subjclass{Primary: 35B65}
\begin{abstract}
In this note we prove an end-point regularity result on the $L^p$ integrability of the second derivatives of solutions to non-divergence form uniformly elliptic equations whose second derivatives are a priori only known to be  integrable. The main assumption on the elliptic operator is the Dini continuity of the coefficients. We provide a counterexample showing that the Dini condition is somehow optimal. We also give a counterexample related to
the $BMO$ regularity of second derivatives of solutions to elliptic equations. These results are analogous to corresponding results for divergence form elliptic equations in \cite{Brezis,mazya}.

\end{abstract}
\maketitle
\section{Introduction}
In this note we investigate some regularity issues for solutions to non-divergence form elliptic equations whose second derivatives are locally integrable. 
Given an open bounded domain $\Omega\subset \R^n$, we will assume that $A(x)=\left(a_{ij}(x)\right)$ is a real symmetric matrix such that there is a $\lambda>0$ verifying
\begin{equation*}
\lambda |\xi|^2\le A(x)\xi\cdot \xi\le \lambda^{-1} \left|\xi\right|^2,\ \text{ for any } \xi\in\R^n,\ x\in \Omega.
\end{equation*}
Here we deal with solutions of operators of the form
\begin{equation} \label{elliptic_operator}
\Lc u=\text{\it tr}\left(A D^2u \right)=\sum_{i,j=1}^n a_{ij}(x)\partial_{ij}u,
\end{equation}
where the entries of the matrix $A$ are continuous functions in $\overline{\Omega}$.

We recall the reader the following regularity fact \cite[Lemma 9.16]{GilbargTrudinger}:
\begin{lemma}\label{mejorafacil} Let $p,q$ be such that $1<p<q<\infty$ and $f$ be in $L^q(\Omega)$. If $u$ in $W_{loc}^{2,p}(\Omega)$ verifies $\Lc u=f$ in $\Omega$, then $u \in W_{loc}^{2,q}(\Omega)$.
\end{lemma}
The previous result does not cover the case $p=1$ and, as far as we know, this case has not been considered in the literature. It is the purpose of this note to deal with it. We remark that Lemma \ref{mejorafacil} is true under the mere assumption of the continuity of the coefficients. However, as we shall see, this mild assumption is not enough in order to improve the integrability of the second derivatives of $W_{loc}^{2,1}$ solutions. On the contrary, a Dini-type condition on the coefficients is sufficient for this purpose and it is optimal. Here we will consider the following Dini condition:
\begin{definition} A function $f:\Omega\subset \R^n\rightarrow \R$  is \textit{Dini continuous} in $\overline{\Omega}$ if there is continuous a non-decreasing function $\theta:[0,+\infty)\rightarrow [0,+\infty)$ verifying 
\begin{equation*}
|f(x)-f(y)|\leq \theta(|x-y|),\ \text{for any}\ x,y\in \Omega
\end{equation*}
and such that
\begin{equation}\label{integralcondition}
 \int_0^1 \frac{\theta(t)}{t}\, dt<+\infty
\end{equation}
and
\begin{equation}\label{doubling}
\theta(2t)\leq 2\theta(t),\ \text{for}\ t\in(0,\tfrac{1}{2}). 
\end{equation}
We will say that $\theta$ is the \textit{Dini modulus of continuity} of $f$.
\end{definition}
Condition \eqref{doubling} is not restrictive. In fact, as we learnt from \cite[Remark 1]{ApushNaz15}, any modulus of continuity satisfying \eqref{integralcondition} can be dominated by $$\tilde{\theta}(t)=t\sup_{\tau\in[t,1]}\frac{\theta(\tau)}{\tau},$$
which is again a Dini modulus of continuity such that $\tilde{\theta}(t)/t$ is non-decreasing. The later implies \eqref{doubling} for $\tilde{\theta}$.

Before stating our results we first briefly review  the case of elliptic equations in divergence form.
In this situation, motivated by a question raised in \cite{Serrin} and the results in \cite{HagerRoss}, H. Brezis proved the following \cite[Theorems 1 and 2]{Brezis} (proofs were published in \cite{Ancona}).
\begin{theorem}\label{T: contracaffarelli}
Let $A$ be a uniformly elliptic matrix such that $A$ is Dini continuous in $\overline{\Omega}$. Let $u$ in $W^{1,1}(\Omega)$ solve
$$\int_{\Omega}A\nabla u \cdot \nabla \varphi\, dx=0,  \ \ \text{for any}\ \varphi\ \text{in}\  C^\infty_0(\Omega).$$
Then, for any $1<p<\infty$, $u$ is in $W_{loc}^{1,p}(\Omega)$ and 
\begin{equation}\label{regularityestimatebrezis}
\|u\|_{W^{1,p}(K)}\leq C \|u\|_{W^{1,1}(\Omega)}
\end{equation}
for any compact subset $K\subset \Omega$,
where $C$ depends on $n$, $p$, $K$, the ellipticity constant, $\Omega$ and the uniform modulus of continuity of the coefficients, but not on the Dini modulus of continuity.
\end{theorem}
The independence of the constant in \eqref{regularityestimatebrezis} with respect to the Dini modulus of continuity by no means implies that this result is true when the coefficients are merely continuous in $\overline{\Omega}$: a counterexample to such assertion is given in \cite{mazya}.

In the context of non-divergence form elliptic equations, the main result proved in this note is the following.
\begin{theorem} \label{Th:improved_regularity} Assume that the coefficients of $\Lc$ are Dini continuous in $\overline{\Omega}$ and let $u$ in $W^{2,1}(\Omega)$ satisfy $\Lc u=f$, a.e. in $\Omega$ with $f$ in $L^{p}(\Omega)$, for some $1<p<\infty$. Then $u$ is in $W_{loc}^{2,p}(\Omega)$ and
$$\|u\|_{W^{2,p}(K)}\leq C \left[\|u\|_{W^{2,1}(\Omega)}+\|f\|_{L^p(\Omega)}\right],$$
for any compact subset $K\subset \Omega$, where $C$ depends on $n$, $p$, $K$, $\lambda$, $\Omega$ and the uniform modulus of continuity of the coefficients, but not on the Dini modulus of continuity.
\end{theorem}
Similarly to the case of divergence form elliptic equations, the Dini condition on $A$ is the optimal to derive such a result. Here we give a counterexample inspired by \cite[Section 3]{Escauriaza94}, showing that Theorem \ref{Th:improved_regularity} is false when the coefficients of $\Lc$ are not Dini continuous.
\begin{theorem} \label{contraejemplo1} There is an operator $\Lc$ with continuous coefficients in $\overline{B}_1$, which are not Dini continuous at $x=0$, and a solution $u$ in $W^{2,1}(B_1)\cap W^{1,1}_0(B_1)$ of $\Lc u=0$ such that $u$ is not in $W^{2,p}(B_{\frac 12})$, for any $p>1$.
\end{theorem}
Concerning the other end-point in the scale of $L^p$ spaces, we recall that the singular integrals theory \cite[Chapter IV]{Stein} allows to prove that weak solutions \cite[Chapter 8]{GilbargTrudinger} to $\Delta u=f$ in $B_2$ have generalized second order derivatives in $BMO(B_1)$ when $f\in L^\infty(B_2)$. Moreover, the Laplace operator can be perturbed in order to obtain similar results for elliptic operators \eqref{elliptic_operator} with  Dini continuous coefficients \cite{changli} or with $A$ verifying
\begin{equation}\label{E: condicion1}
|A(x)-A(y)|\le C/[1+|\log{|x-y|}|],
\end{equation}
for some $C>0$ sufficiently small \cite[Theorem A, ii and Corollary 4.1]{CaffarelliHuang}.

As far as we know, there are no counterexamples in the literature showing that mere continuity of the coefficients is not enough to prove that the second derivatives of solutions of elliptic equations do not belong to BMO in general. The next counterexample, which is a modification of \cite[Proposition 1.6]{mazya}, fills this gap.
\begin{theorem}\label{contraejemplo2} There exists an operator $\Lc$ with continuous coefficients in $\overline{B}_1$, which are not Dini continuous at $x=0$, and a solution $u$ in $W^{2,p}(B_1)\cap W^{1,p}_0(B_1)$ of $\Lc u=0$, $1<p<\infty$,  such that $D^2 u$ is not in $BMO(B_{\frac 12})$.
\end{theorem}
The counterexample in Theorem \ref{contraejemplo2} is sharp because its coefficient matrix $A$ verifies \eqref{E: condicion1} for $x,y$ in $B_1$, for some fixed $C>0$.

The main ingredients in the proof of Theorem \ref{Th:improved_regularity} are the Sobolev inequality and the boundedness of solutions to equations involving the formal adjoint operator $\Lc^*$  given by
$$\Lc^*v=\sum_{i,j=1}^n \p_{ij}(a^{ij}v).$$
In order to make sense of the solutions associated to the operator $\Lc^*$ when the coefficients of $\Lc$ are only continuous we must consider distributional or weak solutions to the \textit{adjoint} equation. For our purposes, we  need to deal with boundary value problems of the form
\begin{equation}\label{adjoint_problem} 
\begin{cases}
\Lc^*w=div^2 \Phi+\eta,\ &\text{in }\Omega,\\
w=\psi+\frac{\Phi\nu\cdot\nu}{A\nu\cdot\nu},\ &\text{on }\p\Omega,
\end{cases}
\end{equation}
where $\Phi=(\varphi^{kl})_{k,l=1}^n$, $div^2\Phi=\sum_{k,l=1}\p_{kl}\varphi^{kl}$, with 
\begin{equation}\label{E: condicones}
\Phi\ \text{in}\  L^p(\Om),\ \eta\ \text{in}\ L^{p}(\Om),\ \psi\ \text{in}\ L^p(\p\Om, d\sigma),\ 1<p<\infty.
\end{equation}
 \begin{definition} Let $\Omega\subset \R^n$ be a bounded $C^{1,1}$ domain with unit exterior normal vector $\nu=(\nu_1,\ldots,\nu_n)$, $\Phi$, $\psi$ and $\eta$ verify \eqref{E: condicones}, let $\Lc$ be as in \eqref{elliptic_operator}, $1<p<\infty$ and $\frac{1}{p}+\frac{1}{p'}=1$.
We say that $w$ in $L^p(\Omega)$ is an \textit{adjoint solution} of \eqref{adjoint_problem} if $w$ satisfies
\begin{equation} \label{weak_formulation}
\begin{split}
\int_{\Omega} w\,\Lc u\,dy&=\int_{\Omega}tr(\Phi D^2 u)\, dy+\int_{\Omega}\eta u\, dy+\int_{\p\Omega} \psi A\nabla u\cdot \nu\, d\sigma(y),
\end{split}
\end{equation}
for any $u$ in $W^{2,p'}(\Omega)\cap W^{1,p'}_0(\Omega)$.
\end{definition}
Later we shall explain why this definition makes sense. At first, the boundary conditions in \eqref{adjoint_problem} may look strange. However, if we formally multiply \eqref{adjoint_problem} by a test function $u$ in $C^\infty(\overline\Omega)$ with $u=0$ on $\partial\Omega$, assume that $v$ is in $C^\infty(\overline\Omega)$ and integrate by parts, taking into account that $\nabla u=(\nabla u\cdot \nu)\nu$ on $\p\Omega$, we arrive at \eqref{weak_formulation}.

We will also consider \textit{local} adjoint solutions of $$\Lc^*w=div^2 \Phi+\eta\ \ \text{ in }\Omega,$$ i.e., solutions which do not satisfy any specified boundary condition. Such local solutions are those in $L^p_{loc}\left(\Omega\right)$ that verify \eqref{weak_formulation}, when $u$ is in $W_0^{2,p'}(\Omega)$; thus, the boundary integrals in \eqref{weak_formulation} are omitted.

This kind of adjoint solutions have been already studied in the literature. For instance, in \cite{Sjogren73, Bauman84,FabesStroock,FGMS, Escauriaza00,Maz'yaMcOwen} solutions of \eqref{adjoint_problem} with $\Phi=0$ are studied under low regularity assumptions on either the coefficients of $\Lc$ or the boundary of the domain. Moreover,  when the data and the boundary of the domain involved in \eqref{adjoint_problem} are smooth, the weak formulation \eqref{weak_formulation} can be recasted in such a way that the regularity theory in \cite{LionsMagenes} or \cite{Necas} can be used to prove that $w$ is smooth and solves \eqref{adjoint_problem} in a classical sense.

For our purposes we need to prove the existence and uniqueness of such adjoint solutions.
\begin{lemma} \label{Lemma_adjointcalderon} Let $1<p<\infty$ and assume that \eqref{E: condicones} holds. Then, there exists a unique adjoint solution $w$ in $L^{p}(\Om)$ of \eqref{adjoint_problem}. Moreover, the following estimate holds
\begin{equation} \label{desicalderonadjoint}
\|w\|_{L^p(\Om)}\leq C\left[\|\Phi\|_{L^p(\Om)}
+\|\eta\|_{L^{p}(\Om)}+\|\psi\|_{L^p(\p \Om)}\right],
\end{equation}
where $C$ depends on $\Omega,p,n,\lambda$ and the continuity of $A$.
\end{lemma}
This result follows from the so-called \textit{transposition} or \text{duality} method \cite{LionsMagenes,Necas}, which relies on the existence and uniqueness of $W^{2,p'}\cap W^{1,p'}_0(\Omega)$ solutions to $\Lc u=f$.

Finally, the proof of Theorem \ref{Th:improved_regularity} requires the boundedness of certain adjoint solutions to problems of the form \eqref{adjoint_problem} with $\Phi=0$. It is at this point where the Dini continuity of the coefficients plays a role. However, and similarly to what it was done in \cite{Brezis}, we only employ the boundedness of these adjoint solutions in a qualitative form, that is, we do not need an specific estimate of the boundedness of those adjoint solutions.

In order to prove the boundedness of the specific adjoint solutions, we employ a perturbative technique based on ideas first stablished in \cite{Caffarelli,CaffarelliPeral} and used in \cite{YYLi16} to prove the continuity of the gradient of solutions to divergence-form second order elliptic systems with Dini continuous coefficients. Accordingly, we do not only prove that those adjoint solutions are bounded but also its continuity.
\begin{lemma} \label{continuityofadjointsolution}
Let $\zeta\in C_0^\infty(B_3)$, $1<p<\infty$ and assume that the elliptic operator $\Lc$ has Dini continuous coefficients in $B_4$. Then, if $v$ in $L^{p}(B_4)$ satisfies
\begin{equation*}
\int_{B_4}v\, \Lc u\, dx=\int_{B_4} \zeta u\, dx,\  \text{for any }\ u\in W^{2,p'}(B_4)\cap W^{1,p'}_0(B_4),
\end{equation*}
$v$ is continuous in $\overline {B}_3$.
\end{lemma}

The paper is organized as follows: in Section \ref{counterexamples} we give the counterexamples stated in Theorems \ref{contraejemplo1} and \ref{contraejemplo2}; in Section \ref{transposition} we prove Lemma \ref{Lemma_adjointcalderon} using the duality method; in Section \ref{proofofcontinuity} we prove that certain adjoint solutions are continuous and in Section \ref{proofofregularity} we prove Theorem \ref{Th:improved_regularity}.

\section{Counterexamples} \label{counterexamples}
In this section we give two counterexamples. Both of them arise as solutions of uniformly elliptic operators of the form
\begin{equation}\label{counterexample_operator}
\Lc_\alpha u=tr\left[\left(I+\alpha(r) \frac{x}{r}\otimes \frac{x}{r} \right)D^2u \right],
\end{equation}
where $(x\otimes x)_{ij}=x_i x_j$, $r=|x|$, with $\alpha$ is a continuous radial function in $\overline{B_1}$, $\alpha(0)=0$.
\begin{proof}[Proof of Theorem \ref{contraejemplo1}]
If we look for a radial solution $u$ of \eqref{counterexample_operator}, we find that $u$ must satisfy
\begin{equation}\label{ecuacioncontraejemplo}
\Lc_\alpha u=(\alpha(r)+1)u''+\frac{n-1}{r}u'=0.
\end{equation}
We choose $$u(r)=\int_r^1 t^{1-n}\left( \log\frac{R}{t}\right)^{-\gamma} dt,\ \ \gamma>1,$$
with $R>1$ to be chosen.
Then
\begin{equation*}
\begin{split}
u'(r)&=-r^{1-n}\left(\log\frac{R}{r}\right)^{-\gamma}\\
u''(r)&=r^{-n}\left(\log \frac{R}{r}\right)^{-\gamma}\left[n-1-\gamma\left(\log\frac{R}{r}\right)^{-1} \right].
\end{split}
\end{equation*}
Hence, $u\in W^{2,1}(B_1)\cap W^{1,1}_0(B_1)$ but $D^2 u\notin L^p(B_1)$ for any $p>1$, when $\gamma >1$ and $R>1$.
Solving \eqref{ecuacioncontraejemplo} for $\alpha$ we obtain
$$\alpha(r)=\frac{\gamma}{(n-1)\log{\frac{R}{r}}-\gamma},$$
which ensures the uniform ellipticity and the continuity of the coefficients of $\Lc_\alpha$ over $\overline B_1$, when $R$ is sufficiently large. However, $\alpha$ is not Dini continuous at $x=0$.
\end{proof}
\begin{proof}[Proof of Theorem \ref{contraejemplo2}]
Let $\varphi\in C^2([0,1])$, $\alpha\in C([0,1])$ and define 
\begin{equation*}
u(x)=x_1 x_2 \varphi(r).
\end{equation*}
A computation shows that
\begin{equation*}
\Lc_\alpha u=\frac{x_1 x_2}{r^2}\left[(n+3)r\varphi'+r^2\varphi''+\alpha(2\varphi+4r\varphi'+r^2\varphi'') \right].
\end{equation*}
Choosing $\varphi(r)=\left(\log\frac{R}{r}\right)^2$ for some $R>1$ yields
$$ \Lc_\alpha u=\frac{x_1 x_2}{r^2}\left[1+\alpha-(2+n+3\alpha)\log{\frac{R}{r}}+\alpha\left(\log\frac{R}{r}\right)^2 \right],$$
which is identically zero in $B_1(0)$ provided that $$\alpha(r)=\frac{(2+n)\log{\frac{R}{r}}-1}{\left(\log\frac{R}{r}\right)^2-3\log\frac{R}{r}+1},$$
and $R>1$ is taken large enough in order to ensure the uniform ellipticity and the continuity of the coefficients of $\Lc_\alpha$ in $\overline B_1$.
A computation shows that
\begin{equation*}
\p_{12}u\geq \frac12 \left(\log\frac{R}{r}\right)^2\ \text{on}\ \overline B_1,
\end{equation*}
when $R>1$ is large enough. Moreover, for any $c\in\R$ there is $\e=\e(c)$ such that $\left(\log\frac{R}{r}\right)^2\geq 4|c|$ in $B_\e$. Thus
$$\int_{B_{\frac 12}} e^{N|\p_{12} u-c|dx}dx \geq \int_{B_\e} e^{\frac{N}{4}(\log\frac{R}{r})^2}dx=+\infty,\ \text{for any}\ N>0,\, c\in\R.$$
By the John-Nirenberg inequality \cite{JohnNirenberg},  $\p_{12}u$ cannot belong to $BMO(B_1)$. 
\end{proof}

\section{Existence of adjoint solutions} \label{transposition}
We recall the following well known existence result for the Dirichlet problem for non-divergence form elliptic equations \cite[Theorem 9.15, Lemma 9.17]{GilbargTrudinger}.
\begin{lemma} \label{CalderonZygmund}
Let $\Om\subset \R^n$ be a $C^{1,1}$ domain, $f$ be in $L^p(\Om)$ and $1<p<\infty$. Then, there exists a unique $u\in W^{2,p}(\Om)\cap W_0^{1,p}(\Om)$ such that $\Lc u=f$ a.e. in $\Om$. Moreover, there is a constant $C>0$ depending on  $\Omega,p,n,\lambda$ and the modulus of continuity of $A$ such that
\begin{equation}\label{desicalderon}
\|u\|_{W^{2,p}(\Om)}\leq C \|f\|_{L^p(\Om)}.
\end{equation}
\end{lemma}
An easy consequence of Lemma \ref{CalderonZygmund} is the existence and uniqueness of adjoint solutions to \eqref{adjoint_problem} stated in Lemma \ref{Lemma_adjointcalderon}. 

\begin{proof}[Proof of Lemma \ref{Lemma_adjointcalderon}] We construct the solution by means of \textit{tranposition}.
If $p'$ is the conjugate exponent of $p$, we define the functional $T:L^{p'}(\Om)\rightarrow \R$ by
\begin{equation}\label{funcional}
T(f)=\int_{\Om}tr\left(\Phi D^2 u \right)\,dx+\int_{\Om} \eta u\,dx+\int_{\p\Om} \psi A\nabla u\cdot \nu\, d\sigma,
\end{equation}
where $u$ in $W^{2,p'}(\Om)\cap W_0^{1,p'}(\Om)$ verifies $\Lc u=f$, a.e. in $\Om$. 
Combining \eqref{desicalderon}, the trace inequality \cite[\S 5.5, Theorem 1]{Evans}, \eqref{funcional} and H\"older's inequality, it is straightforward to check that
$$|T(f)|\leq C\|f\|_{L^{p'}(\Om)} \left[\|\Phi\|_{L^p(\Om)}
+\|\eta\|_{L^{p}(\Om)}+\|\psi\|_{L^p(\p \Om)}\right],$$
where $C=C(A,\Omega,p,n)$.
Hence $T$ is a bounded functional on $L^{p'}(\Om)$ and by the Riesz representation Theorem, there is a unique $w$ in $L^p(\Om)$ such that
\begin{equation} \label{representacion}
T(f)=\int_{\Om} w f \,dx,\  \text{for any}\ f\in L^{p'}(\Om).
\end{equation}
Moreover,
$$\|w\|_{L^p(\Om)}\leq C\left[\|\Phi\|_{L^p(\Om)}
+\|\eta\|_{L^{p}(\Om)}+\|\psi\|_{L^p(\p \Om)} \right].$$
Now, combining \eqref{funcional} and \eqref{representacion}, it is clear that $v$ is the unique adjoint solution to  \eqref{adjoint_problem}.
\end{proof}

\section{Proof of Lemma \ref{continuityofadjointsolution}} \label{proofofcontinuity}
For the proof of Lemma \ref{continuityofadjointsolution} we need first the following Lemma.
\begin{lemma}\label{perturbationlemma}
Let $\Phi\in L^p(B_1)$, $\eta\in L^\infty(B_1)$, $w\in L^p(B_1)$, $1<p<\infty$ and $\Lc$ be an operator like \eqref{elliptic_operator} with continuous coefficients and $A(0)=I$, the identity matrix. Then, if 
\begin{equation*}
\Lc^*w = div^2\Phi+\eta,\  \text{in}\ B_1,
\end{equation*} 
there exists a harmonic function $h$ in $B_{\frac{3}{4}}$ such that
\begin{equation}\label{desigualdelemilla}
\begin{split}
&\|h\|_{L^p(B_{\frac{3}{4}})}\leq M \|w\|_{L^p(B_1)},\\
&\|w-h\|_{L^p(B_{\frac{3}{4}})}\leq M\left[\|\Phi\|_{L^p(B_1)} + \|A-I\|_{L^\infty(B_1)}\|w\|_{L^p(B_1)}+\|\eta\|_{L^{\infty}(B_1)}\right],
\end{split}
\end{equation}
where $M$ depends on $p,n,\lambda$ and the modulus of continuity of $A$.
\end{lemma}
\begin{proof}[Proof of Lemma \ref{perturbationlemma}]
We first prove Lemma \ref{perturbationlemma} assuming that the coefficients of $\Lc$ and  data are smooth in $\overline B_1$. However, the constants in the estimate will only depend on $p$, $\lambda$, $n$ and the modulus of continuity of $A$. Under these assumptions, the regularity theory \cite{Necas,LionsMagenes}, implies that $w$ is smooth in $B_1$. By Fubini's theorem, there is $\frac{3}{4}\le t\le 1$ such that 
\begin{equation}\label{Fubini}
\|w\|_{L^p(\p B_t)}\leq 4^{-\frac 1p}\|w\|_{L^p(B_1)}.
\end{equation}
Using Lemma \ref{Lemma_adjointcalderon} we can find a function $h$ such that
\begin{equation*}
\begin{cases}
\Delta^*h=0,\ &\text{in }B_t,\\
h=w,\ &\text{on }\p B_t,
\end{cases}
\end{equation*}
in the sense of \eqref{adjoint_problem}. Of course, $h$ is harmonic in the interior of $B_t$.
Moreover, the estimate provided by Lemma \ref{Lemma_adjointcalderon} together with \eqref{Fubini} imply
\begin{equation}\label{acotacionarmonica}
\|h\|_{L^p(B_t)}\leq M\|w\|_{L^p(\p B_t)}\leq M4^{-\frac{1}{p}}\|w\|_{L^p(B_1)},
\end{equation}
with $M=M(p,n)$.
Then $w-h$ satisfies
\begin{equation}\label{satisfiedequationlemma}
\begin{split}
\int_{B_t}\left(w-h\right)\Lc udx&=\int_{B_t} tr\left[h(I-A)D^2 u\right]dx+\int_{B_t} tr\left[\Phi D^2 u\right]dx\\
&+\int_{B_t} \eta u dx+\int_{\p B_t} w \left(A-I\right)\nabla u \cdot \nu d\sigma\\
&=\int_{B_t} tr\left[h(I-A)D^2 u\right]dx+\int_{B_t} tr\left[\Phi D^2 u\right]dx\\
&+\int_{B_t} \eta u dx+\int_{\p B_t} w \frac{\left(A-I\right)\nu\cdot \nu}{A\nu\cdot \nu}A\nabla u \cdot \nu d\sigma
\end{split}
\end{equation}
for any $u\in W^{2,p'}(B_t)\cap W_0^{1,p'}(B_t)$. Therefore, $w-h$ is an adjoint solution to a problem which falls into the conditions of Lemma \ref{Lemma_adjointcalderon} and we can apply \eqref{desicalderonadjoint} to the equation \eqref{satisfiedequationlemma} to get that
\begin{multline*}
\|w-h\|_{L^p(B_t)}\leq M \left[ \|A-I\|_{L^\infty(B_t)}\|h\|_{L^p(B_t)} \right. \\ \left. +\|\Phi\|_{L^p(B_t)}+\|A-I\|_{L^\infty(B_t)}\|w\|_{L^p(\p B_t)}+\|\eta\|_{L^p(B_t)} \right],
\end{multline*}
which together with \eqref{acotacionarmonica} imply the desired estimate.
Finally, an approximation argument allows us to derive the same estimate under the more general conditions mentioned above.
\end{proof}

The perturbative technique used in the proof of Lemma \ref{continuityofadjointsolution} is based on the local smallness of certain quantities. We may assume that $A(0)=I$ and that $\theta$ is a Dini modulus of continuity for $A$ on $B_4$. For this reason, if $v$ and $\zeta$ verify the conditions in Lemma \ref{continuityofadjointsolution}, it is handy to define for $0<t,\,\delta\le 1$,
\begin{equation*}
\omega(t)=t^2+\theta(t),\quad \overline{\delta}=M^{-1}\delta^\frac{n}{p}\frac{\omega(\delta)}{1+\|v\|_{L^p(B_1)}+\|\zeta\|_{L^\infty(B_1)}},
\end{equation*}
where $M$ is the constant in \eqref{desigualdelemilla}, and to consider the rescaled functions 
\begin{equation} \label{cantidadesreescaladas}
v_{\delta}(x)=\overline{\delta} v(\delta x),\ \ \ \zeta_\delta(x)=\overline{\delta}\delta^2 \zeta(\delta x).
\end{equation}
From \eqref{doubling}
\begin{equation}\label{doubling2}
\omega(4t)\leq 16\,\omega(t),\ \text{for}\ t\leq 1/4
\end{equation}
and the dilation and rescaling yield 
\begin{equation}\label{unosreescalamientos}
\|v_\delta\|_{L^p(B_1)}\leq M^{-1}\omega(\delta),\ \ \ \|\zeta_\delta\|_{L^\infty(B_1)}\leq M^{-1}\delta^2\omega(\delta).
\end{equation}
Also, 
\begin{equation}\label{queecuacion}
\mathcal L_\delta^\ast v_\delta=\zeta_{\delta},\ \text{in}\ B_1,\quad\text{with}\ \mathcal L_\delta u=tr\left(A(\delta x)D^2u\right).
\end{equation}
Next, we show by induction that there are $C>0$, $0<\delta\le 1$ and harmonic functions $h_k$ in $4^{-k}B_{\frac{3}{4}}$, $k\ge 0$, such that
\begin{equation} \label{induccion}
\begin{split}
&C^{-1}\|h_k\|_{L^p(4^{-k}B_\frac{3}{4})}+\|v-\sum_{j=0}^k h_j\|_{L^p(4^{-k}B_\frac{1}{4})} 
\leq  4^{-k\frac{n}{p}}\omega(4^{-k}\delta),\\
&\|h_k\|_{L^\infty(4^{-k}B_{\frac{1}{2}})}+4^{-k}\|\nabla h_k\|_{L^\infty(4^{-k}B_{\frac{1}{2}})}\leq  C\omega(4^{-k}\delta),
\end{split}
\end{equation}
where $C$ depends on $n,p,\lambda$ and the modulus of continuity of $A$.

When $k=0$, \eqref{unosreescalamientos}, \eqref{queecuacion} and Lemma \ref{perturbationlemma} applied to $v_\delta$ show that there is a harmonic function $h_0$ in $B_{\frac{3}{4}}$ such that
\begin{equation*}
\begin{split}
\|h_0\|_{L^p(B_{\frac{3}{4}})}&\leq M\|v_\delta\|_{L^p(B_1)}\leq   \omega(\delta),\\
\|v_\delta-h_0\|_{L^p(B_{\frac{3}{4}})}&\leq M \left[\theta(\delta)\|v_\delta\|_{L^p(B_1)}+\|\zeta_\delta\|_{L^\infty(B_1)} \right]
\leq  \omega(\delta)^2.
\end{split}
\end{equation*}
By regularity of harmonic functions \cite[\S 2.2.3c]{Evans}
\begin{equation*}
\|h_0\|_{L^{\infty}(B_{\frac{1}{2}})}+\|\nabla h_0\|_{L^\infty(B_{\frac{1}{2}})}\leq C(n,p)\|h_0\|_{L^p(B_{\frac{3}{4}})}\leq C(n,p)\omega(\delta).
\end{equation*}
Thus, \eqref{induccion} holds  for $k=0$, when $C$ and $\delta$ satisfy 
\begin{equation}\label{E:1}
C^{-1}+\omega(\delta)\le 1\ \text{and}\  C\ge C(n,p).
\end{equation}
Now, assume that \eqref{induccion} holds up to some $k\ge 0$ and define 
\begin{equation*}
\begin{split}
A_{k+1}(x)=A(4^{-k-1}\delta x)&,\ \ \ \ \Lc_{k+1}u=tr(A_{k+1}(x) D^2u)\\
G_{k+1}(x)=\left(I-A_{k+1}(x)\right)&\sum_{j=0}^k h_j(4^{-k-1}x).
\end{split}
\end{equation*}
Then, $W_{k+1}(x)=v_\delta(4^{-k-1}x)-\sum_{j=0}^k h_j(4^{-k-1}x)$ solves 
\begin{equation}\label{ecuaciondelauwe}
\Lc_{k+1}^* W_{k+1}(x)=div^2G_{k+1}+4^{-2k-2}\zeta_\delta(4^{-k-1}x),\ \text{in}\ B_1.
\end{equation}
Using the induction hypothesis \eqref{induccion} and \eqref{doubling2} , one finds that $G_{k+1}$ satisfies
\begin{equation}\label{acotaciongmayuscula}
\begin{split}
\|G_{k+1}\|_{L^p(B_1)}&\leq |B_1|^\frac{1}{p} \theta(4^{-k-1}\delta)\sum_{j=0}^k\|h_j(4^{-k-1}\cdot)\|_{L^\infty(B_1)}\\
&\leq \left[32\, C|B_1|^{\frac 1p}\int_0^\delta\frac{\omega(t)}{t}\,dt\right]\theta(4^{-k-1}\delta).
\end{split}
\end{equation}
Besides, the inequality in the first line of \eqref{induccion} gives
\begin{equation}\label{acotacionuwedoble}
\|W_{k+1}\|_{L^p(B_1)}\leq 4^{\frac{n}{p}} \omega(4^{-k}\delta).
\end{equation}
From \eqref{doubling2}, \eqref{ecuaciondelauwe}, \eqref{acotaciongmayuscula} and  \eqref{acotacionuwedoble}, apply Lemma \ref{perturbationlemma} to $W_{k+1}$ to find that with the same $M$, there is a harmonic function $\tilde{h}_{k+1}$ in $B_{\frac{3}{4}}$ such that 
\begin{equation} \label{lados}
\|\tilde{h}_{k+1}\|_{L^p(B_{\frac{3}{4}})}\leq  4^{2+\frac{n}{p}}M\omega(4^{-k-1}\delta).
\end{equation}
and
\begin{equation} \label{launo}
\|W_{k+1}-\tilde{h}_{k+1}\|_{L^p(B_{\frac{3}{4}})} \leq M\left[32\,|B_1|^{\frac 1p}C\int_0^\delta\frac{\omega(t)}t\,dt+\omega(\delta)\right]\omega(4^{-k-1}\delta).
\end{equation}
From standard interior estimates for harmonic functions and \eqref{lados}
\begin{equation}\label{latres}
\|\tilde{h}_{k+1}\|_{L^\infty(B_\frac{1}{2})}+\|\nabla \tilde{h}_{k+1}\|_{L^\infty(B_{\frac{1}{2}})}\le C(n,p) 4^{2+\frac{n}{p}}M\omega(4^{-k-1}\delta).
\end{equation}
Setting, $h_{k+1}(x)=\tilde{h}_{k+1}(4^{k+1}x)$, the last three formulae and \eqref{E:1} show that the induction hypothesis holds when $C=2\,C(n,p)\left[4^{2+\frac np}M+1\right]$ and $\delta$ is determined by the condition
\begin{equation*}
2M\left[32\, |B_1|^{\frac 1p} C\int_0^\delta\frac{\omega(t)}t\,dt+\omega(\delta)\right]\le 1.
\end{equation*}
On the other hand, for $|x|\leq 4^{-k-1}$
\begin{equation}\label{acotacionpuntual}
\begin{split}
|\sum_{j=0}^k h_j(x)-\sum_{j=0}^\infty h_j(0)|&\leq \sum_{j=k+1}^\infty|h_j(0)|+4^{-k-1}\sum_{j=0}^k\|\nabla h_j\|_{L^\infty(4^{-k}B_{\frac{1}{4}})}\\
&\leq 16\, C\left(\int_0^{4^{-k}\delta}\frac{\omega(t)}{t}\,dt+4^{-k-1}\delta\int_{4^{-k-1}\delta}^\delta\frac{\omega(t)}{t^2}\,dt \right)\\
\end{split}
\end{equation}
Therefore, \eqref{induccion} together with \eqref{acotacionpuntual} and \eqref{doubling2} imply
\begin{multline}\label{implicacontinuidad}
\text{\rlap |{$\int_{4^{-k-1}B_1}$}}|v_\delta (x)-\sum_{j=0}^\infty h_j(0)|\,dx\le \\ \leq 4^4 C\left[\int_0^{4^{-k-1}\delta}\frac{\omega(t)}{t}dt+4^{-k-1}\delta\int_{4^{-k-1}\delta}^\delta\frac{\omega(t)}{t^2}dt+\omega(4^{-k-1}\delta)\right],
\end{multline}
when $k\ge 0$. Using Fubini's theorem it is easy to check that $t\int_{t}^1 \frac{\omega(s)}{s^2}\,ds$ is a Dini modulus of continuity, one can verify that
\begin{equation*}
\sigma(t)=\int_0^t\frac{\omega(s)}s\,ds+t\int_t^1\frac{\omega(s)}{s^2}\,ds+\omega(t)
\end{equation*}
is non-decreasing and derive that
$\lim_{t\to 0^+}\sigma(t)\rightarrow 0$. Hence, from \eqref{implicacontinuidad} and \eqref{cantidadesreescaladas}, we have proved that there are $C>0$, depending on $\lambda$, $n$ and the modulus of continuity Dini of $A$, and a number $a(0)$ such that
\begin{equation}\label{implicacontinuidad1}
\text{\rlap |{$\int_{B_r}$}}|v(x)-a(0)|dx\leq C \sigma(r)\left[\|v\|_{L^p(B_1)}+\|\zeta\|_{L^\infty(B_1)}\right],\ \text{when}\ 0<r\le 1.
\end{equation}
Since $v\in L^p(B_4)$ is an adjoint solution in $B_4$, we can repeat the proof of \eqref{implicacontinuidad1} in balls of radius $1$ centered at any point $\xb\in B_3$. We note that the constant $C$ and the modulus of continuity $\sigma$ in \eqref{implicacontinuidad} do not depend on the center of the ball, hence, for each $\xb\in B_3$ we can find a number $a(\xb)$ such that
\begin{equation}\label{implicacontinuidad2}
\text{\rlap |{$\int_{B_r(\xb)}$}}|v(x)-a(\xb)|dx\leq C \sigma(r)\left[\|v\|_{L^p(B_4)}+\|\zeta\|_{L^\infty(B_4)}\right],\ \text{when }\ 0<r\le 1.
\end{equation}
By Lebesgue's differentiation theorem, $u$ and $a$ are equal a.e. in $B_3$. Now, if $\xb$ and $\yb$ are in $B_3$ and $\frac r2\leq |\xb-\yb|\leq r$, we have
\begin{equation*}
\begin{split}
|u(\xb)-u(\yb)|&\le \text{\rlap |{$\int_{B_r(\xb)}$}}|u(\xb)-u(x)|dx+\text{\rlap |{$\int_{B_r(\xb)}$}}|u(x)-u(\yb)|dx\\
&\lesssim \text{\rlap |{$\int_{B_r(\xb)}$}}|u(\xb)-u(x)|dx+\text{\rlap |{$\int_{B_r(\yb)}$}}|u(x)-u(\yb)|dx\\
& \lesssim \sigma(2r)\left[\|v\|_{L^p(B_4)}+\|\zeta\|_{L^\infty(B_4)}\right],\ \text{when }\ 0<r\le 1/2.
\end{split}
\end{equation*}
which proves  Lemma \ref{continuityofadjointsolution}.

\section{Proof of Theorem \ref{Th:improved_regularity}} \label{proofofregularity}

It suffices to show that if $u$ in $W^{2,1}(B_4)$ verifies $\Lc u=f$, with $f$ in $L^{p}(B_4)$, $1<p<\infty$, then $u\in W^{2,q}(B_1)$, for some $q>1$. Let then $\eta$ be a function in $C_0^{\infty}(B_2)$ with $\eta=1$ in $B_1$ and $0\leq \eta\leq 1$. Set $q=\min\left\lbrace\frac{n}{n-1},p\right\rbrace$ and let $\varphi$ be in $C^\infty_0(B_3)$ with $\|\varphi\|_{L^{q'}(B_3)}\leq 1$.
We shall show that
\begin{equation}\label{acotaciondualidad}
\left|\int_{B_4}\p_{kl}(u\eta)\varphi\,dx\right|\leq C \left[\|f\|_{L^p(B_4)}+\|u\|_{W^{2,1}(B_4)}\right],
\end{equation}
where $C$ only depends on $q$, $p$, $\lambda$, $n$ and the uniform modulus of continuity of the coefficients $A$, but not on the Dini modulus of continuity of $A$. 

Let $u_\e$ in $C^\infty(B_4)$ be a sequence of functions converging to $u$ in $W^{2,1}_{loc}(B_4)$ as $\e \rightarrow 0$, then  for any $\varphi$ in $C_0^{\infty}(B_3)$ we have
$$\int_{B_4}\p_{kl}(u \eta)\varphi\,dx=\lim_{\e\rightarrow 0}\int_{B_4}\p_{kl}(u_\e \eta)\varphi\,dx.$$

By Lemma \ref{Lemma_adjointcalderon} with $\Omega=B_4$ and $p=q'$, for $k,l\in\{1,\dots,n\}$, there is a unique weak adjoint solution $v$ in $L^{q'}(B_4)$ to
\begin{equation*}
\begin{cases}
\Lc^*v=\partial_{kl}\varphi,\ &\text{on}\ B_4,\\
v=0,\ &\text{on}\ \p B_4.
\end{cases}
\end{equation*}
That is, a function $v$ in $L^{q'}(B_4)$ such that
\begin{equation*}
\begin{split}
\int_{B_4} v\,\Lc w\,dy&=\int_{B_4}\varphi\,\partial_{kl} w\, dy,
\end{split}
\end{equation*}
for any $w$ in $W^{2,q}(B_4)\cap W^{1,q}_0(B_4)$ and
\begin{equation}\label{acotacionLqadjoint}
\|v\|_{L^{q'}(B_4)}\leq C \|\varphi\|_{L^{q'}(B_3)}\leq C.
\end{equation} 
Observe that $u_\e\eta$ is in $W^{2,q}(B_4)\cap W^{1,q}_0(B_4)$,  for any $\e>0$. Thus, 
\begin{equation}\label{ojoconestaindentidad}
\int_{B_4}\p_{kl}(u_\e \eta)\varphi\,dx= \int_{B_4}v\,\mathcal L(u_\e \eta)\,dx. 
\end{equation}

Now, we want to take limits in \eqref{ojoconestaindentidad} as $\e\rightarrow 0$. A priori, we only know that $\p_{kl}u$ is in $L^1(B_4)$, so we can just assert that $\Lc(u_\e\eta)\rightarrow\Lc(u\eta)$ in $L^1(B_4)$ as $\e\rightarrow 0$. However, in order to take the limit as $\e\rightarrow 0$ inside of the integral in the right-hand side of \eqref{ojoconestaindentidad} and because of the support properties of the functions involved, we only need to know that $v$ is bounded in $\overline B_3$, which indeed is the case because of Lemma \ref{continuityofadjointsolution}, with $\zeta=\partial_{kl}\varphi$.
Hence, we obtain
\begin{equation*}
\begin{split}
\int_{B_4}\p_{kl}(u \eta)\varphi\,dx&=\int_{B_4}v\,\Lc(u\eta)\,dx=\int_{B_4}v\eta\,\Lc u\,dx+\int_{B_4}v u\,\Lc\eta \,dx\\
&+2\int_{B_4}v A \nabla u \cdot \nabla\eta\,dx\triangleq J_1+J_2+J_3.
\end{split}
\end{equation*}
Now, H\"older's inequality, Sobolev's inequality and \eqref{acotacionLqadjoint} yield
\begin{equation*}
\begin{split}
|J_1|&\leq  \|v\|_{L^{q'}(B_4)} \|\Lc u\|_{L^q(B_4)}\leq C \|f\|_{L^p(B_4)},\\
|J_2|&\leq M \|v\|_{L^{q'}(B_4)} \|u\|_{L^q(B_4)}\leq C \|u\|_{W^{1,1}(B_4)},\\
|J_3|&\leq M \|v\|_{L^{q'}(B_4)} \|\nabla u\|_{L^q(B_4)}\leq C \|u\|_{W^{2,1}(B_4)},
\end{split}
\end{equation*}
which implies \eqref{acotaciondualidad}, and by density and duality
\begin{equation*}
\|\p_{kl}(u\eta)\|_{L^q(B_3)}\leq C \left[\|f\|_{L^p(B_4)}+\|u\|_{W^{2,1}(B_4)}\right].
\end{equation*}
Therefore, $u$ is in $W^{2,q}(B_1)$ and 
$$\|u\|_{W^{2,q}(B_1)}\leq C \left[\|f\|_{L^p(B_4)}+\|u\|_{W^{2,1}(B_4)}\right],$$
which is the desired estimate.

\bibliographystyle{plain}
\bibliography{bibliografia}{}

\begin{thebibliography}{10}

\bibitem{Ancona}
A.~Ancona.
\newblock {\em Elliptic operators, conormal derivatives and positive parts of functions. With an appendix by Ha\"im Brezis. }
\newblock {J. Funct. Anal. {\bf 257}, 7 (2009), 2124--2158.}


\bibitem{ApushNaz15}
 D.E.~Apushkinskaya, A.I.~Nazarov.
\newblock {\em A counterexample to the Hopf-Oleinik lemma (elliptic case).}
\newblock {Preprint (2015),  arXiv:1503.02179v2 [math.AP].}

\bibitem{Bauman84}
P.~Bauman.
\newblock {\em Positive solutions of elliptic equations in nondivergence form and
  their adjoints}.
\newblock {Ark. Mat. {\bf 22}, 2 (1984), 153--173.}

\bibitem{Brezis}
H.~Brezis.
\newblock {\em On a conjecture of {J}. {S}errin.}
\newblock {Atti Accad. Naz. Lincei Cl. Sci. Fis. Mat. Natur. Rend. Lincei (9) Mat. Appl. {\bf 19}, 4 (2008) 335--338.}

\bibitem{Caffarelli}
L.~Caffarelli.
\newblock {\em Interior a Priori Estimates for Solutions of Fully Non-Linear Equations}.
\newblock {Ann. Math. {\bf 130}, 1 (1989) 189--213.}

\bibitem{CaffarelliHuang}
L.~Caffarelli, Q.~Huang.
\newblock {\em Estimates in the generalized Campanato-John-Nirenberg spaces for fully nonlinear elliptic equations}.
\newblock {Duke Math. J. {\bf 118}, 1 (2003) 1--17.}

\bibitem{CaffarelliPeral}
L.~Caffarelli, I.~Peral.
\newblock {\em On $W^{1,p}$ estimates for elliptic equations in divergence form.}
\newblock {Comm. Pure Appl. Math. {\bf 51}, 1 (1998) 1--21.}

\bibitem{changli}
D.-C.~Chang, S.-Y.~Li.
\newblock {\em On the boundedness of multipliers, commutators and the second
  derivatives of {G}reen's operators on {$H^1$} and {BMO}.}
\newblock {Ann. Scuola Norm. Sup. Pisa Cl. Sci. (4) {\bf 28}, 2 (1999) 341--356.}


\bibitem{Escauriaza94}
L.~Escauriaza.
\newblock {\em Weak type-{$(1,1)$} inequalities and regularity properties of adjoint
  and normalized adjoint solutions to linear nondivergence form operators with
  {VMO} coefficients.}
\newblock {Duke Math. J. {\bf 74}, 1 (1994) 177--201.}

\bibitem{Escauriaza00}
L.~Escauriaza.
\newblock {\em Bounds for the fundamental solution of elliptic and parabolic
  equations in nondivergence form.}
\newblock {Comm. Partial Differential Equations {\bf 25}, 5-6 (2000) 821--845.}

\bibitem{Evans}
L.C.~Evans.
\newblock {\em Partial differential equations}, Vol. 19 {\em Graduate
  Studies in Mathematics}.
\newblock American Mathematical Society, Providence, RI, 1998.

\bibitem{FGMS}
E.B.~Fabes, N.~Garofalo, S.~Marin-Malave, S.~Salsa.
\newblock {\em Fatou theorems for some nonlinear elliptic equations.}
\newblock {Rev. Mat. Iberoamericana {\bf 4}, 2 (1988) 227--251.}

\bibitem{FabesStroock}
E.B.~Fabes, D.W.~Stroock.
\newblock {\em The {$L^p$}-integrability of {G}reen's functions and fundamental
  solutions for elliptic and parabolic equations.}
\newblock {Duke Math. J. {\bf 51}, 4 (1984) 997--1016.}

\bibitem{GilbargTrudinger}
D.~Gilbarg, N.S.~Trudinger.
\newblock {\em Elliptic partial differential equations of second order}.
\newblock Classics in Mathematics. Springer-Verlag, Berlin, 2001.
\newblock Reprint of the 1998 edition.

\bibitem{HagerRoss}
R.A.~Hager, J.~Ross.
\newblock {\em A regularity theorem for linear second order elliptic divergence
  equations.}
\newblock {Ann. Scuola Norm. Sup. Pisa (3) {\bf 26} (1972) 283--290.}

\bibitem{mazya}
T.~Jin, V.~Maz'ya, J.~Van~Schaftingen.
\newblock {\em Pathological solutions to elliptic problems in divergence form with
  continuous coefficients}.
\newblock {C. R. Math. Acad. Sci. Paris {\bf 347}, (13-14) (2009)  773--778.}

\bibitem{JohnNirenberg}
F.~John, L.~Nirenberg.
\newblock {\em On functions of bounded mean oscillation}.
\newblock {Commun. Pur. Appl. Math. {\bf 14} (1961) 415--426.}

\bibitem{YYLi16}
Y.~Li.
\newblock {\em On the $C^1$ regularity of solutions to divergence form elliptic
  systems with Dini-continuous coefficients.}
\newblock {Preprint (2016), arXiv:1605.00535 [math.AP].}

\bibitem{LionsMagenes}
E.~Magenes, J.-L.~Lions.
\newblock {\em Non-homogeneous boundary value problems and applications. {V}ol.
  {I}}.
\newblock Springer-Verlag, New York-Heidelberg, 1972.
\newblock Translated from the French by P. Kenneth, Die Grundlehren der
  mathematischen Wissenschaften, Band 181.

\bibitem{Maz'yaMcOwen}
V.~Maz'ya, R.~McOwen.
\newblock {\em Asymptotics for Solutions of Elliptic Equations in Double Divergence Form.}
\newblock {Commun. Part. Diff. Eq.  {\bf 32}, 2 (2007) 191--207.}
  
\bibitem{Necas}
J.~Ne{\v{c}}as.
\newblock {\em Direct methods in the theory of elliptic equations}.
\newblock Springer Monographs in Mathematics. Springer, Heidelberg, 2012.
\newblock Translated from the 1967 French original by Gerard Tronel and Alois Kufner. 

\bibitem{Serrin}
J.~Serrin.
\newblock {\em Pathological solutions of elliptic differential equations.}
\newblock {Ann. Scuola Norm. Sup. Pisa (3) {\bf 18} (1964) 385--387.}

\bibitem{Sjogren73}
P.~Sj{\"o}gren.
\newblock {\em On the adjoint of an elliptic linear differential operator and its
  potential theory.}
\newblock {Ark. Mat. {\bf 11} (1973) 153--165.}

\bibitem{Stein}
E.M.~Stein.
\newblock {\em Harmonic analysis: real-variable methods, orthogonality, and
  oscillatory integrals}.
\newblock With the assistance of Timothy S. Murphy. Princeton Mathematical Series, 43. Monographs in Harmonic Analysis, III. Princeton University Press, Princeton, NJ, 1993. xiv+695 pp.






\end{thebibliography}
\end{document}